\documentclass[12pt]{amsart}

\usepackage{amssymb,amsthm}
\usepackage{amsfonts,cmtiup,comment,stmaryrd}

\newtheorem{theorem}{Theorem}[section]
\newtheorem{lemma}[theorem]{Lemma}
\newtheorem{proposition}[theorem]{Proposition}
\newtheorem{corollary}[theorem]{Corollary}

\theoremstyle{definition}
\newtheorem{definition}[theorem]{Definition}
\newtheorem{remark}[theorem]{Remark}
\newtheorem*{definition*}{Definition}
\newtheorem*{notat}{Notation}
\newtheorem*{Index Convention}{Index Convention}

\numberwithin{equation}{section}

\begin{document}
\title{Finite groups and Lie rings\\ with an automorphism of order \boldmath  $2^n$}

\author{E. I. Khukhro}
\address{Sobolev Institute of Mathematics, Novosibirsk, 630\,090,
Russia,\newline and University of Lincoln, U.K.} \email{khukhro@yahoo.co.uk}

\author{N. Yu. Makarenko}
\address{Sobolev Institute of Mathematics, Novosibirsk, 630\,090,
Russia}\email{natalia\_makarenko@yahoo.fr} 

\author{P. Shumyatsky}
\address{Department of Mathematics, University of Brasilia, DF~70910-900, Brazil}
\email{pavel@unb.br}

\keywords{Finite group, Lie ring, derived length, Fitting height,
nilpotency class, automorphism}
\subjclass{Primary 20D45; secondary 17B40, 20F40, 20F50}

\begin{abstract}
Suppose that a finite group $G$ admits an automorphism $\varphi$ of
order $2^n$ such that the fixed-point subgroup $C_G(\varphi 
^{2^{n-1}})$ of the involution $\varphi ^{2^{n-1}}$
is nilpotent of
class $c$. Let $m=|C_G(\varphi )|$ be the number of fixed points of
$\varphi$. It is proved that $G$ has a characteristic soluble subgroup of derived
length bounded in terms of  $n,c$ whose index is bounded in terms
of $m,n,c$. A similar result is also proved for Lie rings.
\end{abstract}

\maketitle

\section{Introduction}
\baselineskip15pt Suppose that a finite group $G$ admits an
automorphism $\varphi$.
It follows from the classification
of finite simple groups that if $\varphi$ is fixed-point-free, that is,
$C_G(\varphi )=1$, then $G$ is soluble \cite{row}, and when in
addition $|\varphi |$ is a prime, $G$ is nilpotent by Thompson's
theorem \cite{tho59} (which does not use the classification but
rather lies in its foundation).  Extending the Brauer--Fowler
theorem, using the classification Hartley \cite{har} proved that if $|C_G(\varphi )|=m$, then
$G$ has a soluble subgroup of $(|\varphi |,m)$-bounded index. (Henceforth
we write, say, ``$(a,b,\dots )$-bounded'' to abbreviate
``bounded above in terms of  $a, b,\dots
$ only''.)

Now let $G$ be soluble from the outset; further results were
obtained about the Fitting height (the length of a shortest normal
series with nilpotent factors). When $C_G(\varphi )=1$, by a special
case of Dade's theorem \cite{dad} the Fitting height of $G$ is
bounded in terms of $\alpha (|\varphi |)$ --- the number of prime factors
of $|\varphi |$ counting multiplicities; the coprime case $(|G|,|\varphi 
|)=1$ of this result was proved earlier as a special case of
Thompson's theorem \cite{tho64}. In the general situation, when
$|C_G(\varphi )|=m$,  it is conjectured that $G$ has a subgroup of
$(|\varphi |,m)$-bounded index with Fitting height bounded  in terms of
$\alpha (|\varphi |)$. This conjecture was proved in the coprime case by
Hartley and Isaacs \cite{ha-is} using Turull's results \cite{tu}, and in the case where $|\varphi |$ is a
prime-power by Hartley and Turau \cite{ha-tu}. A weaker  bound for the Fitting height,
in terms of $|\varphi |$ and $m$ was also obtained in the case
where $|\varphi |$ is a product of two prime-powers in an unpublished
note by Hartley \cite{ha-unp}.  (The aforementioned results of  Thompson \cite{tho64}, Hartley--Issacs \cite{ha-is}, and Turull \cite{tu} are actually about any,  not necessarily cyclic, soluble groups of automorphisms, and Dade's theorem \cite{dad} is about any Carter subgroup.)

When there is a bound for the Fitting height, further studies are
naturally reduced to nilpotent groups. It is conjectured that if
$C_G(\varphi )=1$, then the derived length of $G$ is bounded in terms
of $|\varphi |$. So far this is proved only when $|\varphi |$ is a prime due
to Higman \cite{hi} (and Kreknin--Kostrikin \cite{kr,kr-ko} with
an explicit bound), or $|\varphi |=4$ due to Kov\'acs \cite{kov}. In
these two cases even the `almost fixed-point-free' theorems were
proved by  Khukhro \cite{khu90}, and Khukhro and Makarenko
\cite{khu-mak-4}, which give a subgroup of $|\varphi |$-bounded
nilpotency class or derived length with $(|\varphi |,|C_G(\varphi 
)|)$-bounded index.
 Another area where definitive results of this kind were proved is the case where $G$ is a finite $p$-group and $|\varphi |$ is a power of $p$
 (Alperin \cite{alp}, Khukhro \cite{khu85}, Shalev \cite{sha93}, Khukhro \cite{khu93}, Medvedev \cite{med}, Jaikin-Zapirain \cite{jai}).

All these results on nilpotent groups are based on the
corresponding theorems on automorphisms of Lie rings. In
particular, by Kreknin's theorem \cite{kr} a Lie ring $L$ with a fixed-point-free
automorphism of finite order $n$ is soluble of $n$-bounded derived length.
Khukhro and Makarenko \cite{khu-mak04} also proved almost
solubility of a Lie algebra (or a Lie ring, under some additional conditions, which hold, for example, for finite Lie rings)  with an almost
fixed-point-free automorphism $\varphi$ of finite order, with a
`strong'  bound, in terms of $|\varphi |$ only, for the derived length
of a subalgebra (or a subring) of bounded codimension (or index in the additive
group).  But group-theoretic analogues of
these results remain open conjectures, except for the cases where $|\varphi |$ is a prime or 4, as described above.

Therefore it makes sense to obtain results in this direction under
additional conditions. One  such result was obtained by Shumyatsky \cite{shu01}: if a finite group $G$ admits a
fixed-point-free automorphism $\varphi$ of order $2^n$ such that the
fixed-point subgroup $C_G(\varphi ^{2^{n-1}})$ of the involution $\varphi 
^{2^{n-1}}$ is nilpotent of class $c$, then $G$ is soluble of
$(n,c)$-bounded derived length. The purpose of the present paper
is an `almost fixed-point-free' generalization of this result.

\begin{theorem}\label{t1} Suppose that a finite group $G$ admits an
automorphism $\varphi$ of order $2^n$ such that the  fixed-point subgroup $C_G(\varphi ^{2^{n-1}})$ of the involution
$\varphi ^{2^{n-1}}$ is nilpotent of class~$c$. Let $m=|C_G(\varphi  
)|$ be the number of fixed points of $\varphi$. Then $G$ has a characteristic soluble
subgroup of $(m,n,c)$-bounded index that has $(n,c)$-bounded
derived length. \end{theorem}

In fact, the condition in the theorem that $C_G(\varphi ^{2^{n-1}})$
is nilpotent of class  $c$ can be weakened to requiring all Sylow subgroups of
$C_G(\varphi ^{2^{n-1}})$ to be nilpotent of class at most~$c$; see Remark~\ref{r1}.
The standard inverse limit argument yields a consequence for
locally finite groups.

\begin{corollary}\label{c1}
 Suppose that a locally finite group $G$ contains an
element $g$ of order $2^n$ with finite centralizer of order
$m=|C_G(g )|$ such that the centralizer
$C_G(g ^{2^{n-1}})$ of the involution $g ^{2^{n-1}}$ is nilpotent
of class $c$.  Then $G$ has a  characteristic  soluble subgroup of finite
$(m,n,c)$-bounded index that has $(n,c)$-bounded derived length.
 \end{corollary}

Here, too,  the condition that $C_G(g ^{2^{n-1}})$
is nilpotent of class  $c$ can be weakened to requiring all nilpotent subgroups of
$C_G(g ^{2^{n-1}})$ to be nilpotent of class at most $c$.

In our recent paper \cite{kms145} we also used Theorem~\ref{t1} to prove that if a 
locally finite group $G$ has a $2$-element $g$ with Chernikov centralizer such that 
the involution in $\langle g\rangle$ has nilpotent centralizer,  then $G$ has a soluble subgroup of finite index.

By the aforementioned results the proof of Theorem~\ref{t1} reduces to the case
of nilpotent groups, where a Lie ring method of `graded
centralizes' developed in \cite{khu90,khu-mak04} is used in
conjunction with ideas of the proof in \cite{shu01}.
We state separately the corresponding Lie ring result, which is used in the
proof of Theorem~\ref{t1}.

\begin{theorem}\label{t2} Suppose that a finite
Lie ring $L$ admits an automorphism $\varphi$ of order $2^n$ such
that the fixed-point subring $C_L(\varphi ^{2^{n-1}})$ of the
involution $\varphi ^{2^{n-1}}$ is nilpotent of class $c$. Let
$m=|C_L(\varphi )|$ be the number of fixed points of $\varphi$. Then $L$ has
ideals $M_1\geqslant M_2$ such that $M_1$ has $(m,n)$-bounded index in
the additive group $L$, the quotient $M_1/M_2$ is nilpotent of class at most $c+1$,
and $M_2$ is nilpotent of  $(n,c)$-bounded class. \end{theorem}

Theorem~\ref{t2} can be extended to Lie algebras over a field  and
to other classes of Lie rings admitting such an automorphism of
order $2^n$. Here we confine ourselves to the case of finite
Lie rings, since this is sufficient for the
purpose of proving Theorem~\ref{t1}. Even in view of the aforementioned general
Khukhro--Makarenko theorem \cite{khu-mak04}, Theorem~\ref{t2}  still makes sense, since it gives a stronger `metanilpotent' conclusion (of course, under stronger assumptions).

In \S\,\ref{s-prelim} we give definitions, introduce notation, and
list several results that are used in the sequel. In
\S\,\ref{s-rings} we prove Theorem~\ref{t2} on Lie rings using a
modification of the method of graded centralizers developed in
\cite{khu90,khu-mak04} for studying almost fixed-point-free
automorphisms. Theorem~\ref{t1} is proved in
\S\,\ref{s-groups}. Known
results reduce the proof to the case of a nilpotent
group. Then we firstly apply the Lie ring method similarly to \cite{shu01} to obtain a `weak' bound, depending on $m,n,c$, for the nilpotency class of $[G,\varphi ^{2^{n-1}}]$. Finally, Theorem~\ref{t2}, or rather one of the
propositions in its proof, is used to obtain the required `strong' bound, in terms of $n,c$ only, for the nilpotency class of $[H,\varphi ^{2^{n-1}}]$ for a certain subgroup $H$ of $(m,n,c)$-bounded index. When a subgroup of $(m,n,c)$-bounded index and of $(n,c)$-bounded derived length is constructed, we obtain a characteristic subgroup of $(m,n,c)$-bounded index and of the same derived length due to the general result \cite{khu-mak-char} on subgroups of finite index satisfying a multilinear commutator law; see  Theorem~\ref{t-char}.

\section{Preliminaries} \label{s-prelim}

First we recall some definitions and notation. Products in a Lie
ring are called commutators. A
simple commutator $[a_1,a_2,\dots ,a_s]$ of weight (length)
 $s$ is the commutator $[...[[a_1,a_2],a_3],\dots ,a_s]$. The Lie subring and the ideal
generated by a subset~$S$ are denoted by $\langle S\rangle $  and
${}_{{\rm id}}\!\left< S \right>$, respectively. For additive subgroups $U,V$ of a Lie ring, $[U,V]$ denotes the additive subgroup generated by all commutators $[u,v]$, $u\in U$, $v\in V$. Terms of the lower central series of a Lie ring $L$ start from
$\gamma_1(L)=L$, and by induction,  $\gamma_{i+1}(L)=[\gamma_i(L),L]$. A Lie ring
$L$ is nilpotent of class at most~$h$ if $\gamma_{h+1}(L)=0$. Terms of the derived series start from $L=L^{(0)}$, and by induction, $L^{(i+1)}=[L^{(i)},L^{(i)}]$. A Lie ring $L$ is soluble of derived length at most $d$ if $L^{(d)}=0$.

Let $A$ be an additively written abelian group. A Lie ring $L$ is
\textit{$A$-graded} if
$$L=\bigoplus_{a\in A}L_a\qquad \text{ and }\qquad[L_a,L_b]\subseteq L_{a+b},\quad a,b\in A,$$
where the grading components $L_a$ are subgroups of the additive group of~$L$.
Elements of the $L_a$ are called \textit{homogeneous} (with
respect to this grading), and commutators in homogeneous elements
\textit{homogeneous commutators}. A subgroup
 $H$ of the additive group of $L$ is said to be \textit{homogeneous}
if $H=\bigoplus_{a\in A} (H\cap L_a)$; then we set $H_a=H\cap
L_a$. Obviously, any subring or an ideal generated by homogeneous
additive subgroups is
 homogeneous. A homogeneous subring and the
quotient ring by a homogeneous ideal can be regarded as $A$-graded
rings with the induced gradings.

\begin{Index Convention} For a homogeneous element of a  $({\Bbb Z}/n{\Bbb Z})$-graded Lie ring $L$ we use a small letter
with an index that only indicates the grading component to which
this element belongs: $x_i\in L_i$. Thus, different elements can
be denoted by the same symbol, since it will only matter to which
component these elements belong. For example, $x_1$ and
$x_1$ can be different elements of $L_1$, so that $[x_1,\, x_1]$
can be a nonzero element of $L_2$. These indices are considered
modulo~$n$; for example, $a_{-i}\in L_{-i}=L_{n-i}$.
\end{Index Convention}

Note that under the Index Convention a homogeneous commutator
 belongs to the component $L_s$, where  $s$ is the sum modulo $n$ of the indices of all the elements occurring in this commutator.

Suppose that a Lie ring $L$ admits an automorphism
$\varphi $ of order $n$. Let $\omega$ be a
primitive $n$-th root of unity. We extend the ground ring by
$\omega$ and denote by $\widetilde L$ the ring $L\otimes _{{\Bbb
Z} }{\Bbb Z} [\omega ]$. Then $\varphi $ naturally acts on
$\widetilde L$ and, in particular, $C_{\widetilde L}(\varphi  ) =
C_L(\varphi )\otimes  _{{\Bbb Z} }{\Bbb Z} [\omega ]$.

 We define the analogues of eigenspaces $L_k$ for
$k=0,\,1,\,\ldots ,n-1$ as
$$
L_k=\big\{ a\in \widetilde L\mid a^{\varphi}=\omega ^{k}a\big\} .
$$
If $n$ is invertible in the ground ring of $L$ (for example,
when $L$ is finite of order coprime to $n$), then
$$
\widetilde L= L_0 \oplus L_1 \oplus  \dots \oplus L_{n-1}
$$
(see, for example,~\cite[Ch.~10]{hpbl}). This is a $({\Bbb
Z}/n{\Bbb Z})$-grading because
$$
[L_s,\, L_t]\subseteq L_{s+t\,({\rm mod}\,n)}\quad \text{for all}\;\, s,t.$$

\begin{notat}
Whenever we say that $L_0\oplus L_1\oplus \cdots\oplus L_{n-1}$ is a $({\Bbb Z} /n{\Bbb Z} )$-graded Lie ring, we mean that the $L_i$ are the grading components,  so that $[L_s,\, L_t]\subseteq L_{s+t\;({\rm
mod}\;n)}$.
\end{notat}

We now state the `graded' version of the Khukhro--Makarenko theorem \cite{khu-mak04} on Lie rings with an almost fixed-point-free automorphism of finite order.

\begin{theorem}[{\cite[Corollary~2]{khu-mak04}}]\label{t3} Suppose that $L=
L_0\oplus L_1\oplus \cdots\oplus L_{n-1}$ is a $({\Bbb Z} /n{\Bbb
Z} )$-graded Lie ring. If the component $L_0$ is finite of order $m$, then
$L$ has a soluble homogenous
 ideal $M$ of $n$-bounded derived length and of finite $(m,n)$-bounded index in the additive group of~$L$.
\end{theorem}

We now introduce specialized notation for our case of an automorphism of order $2^n$. Let $L= L_0 \oplus L_1\oplus  \dots \oplus L_{2^n-1} $ be a $({\Bbb Z}/2^n{\Bbb Z})$-graded Lie ring.

\begin{notat}
Let $L_{\rm odd}$ denote the set of all `odd' grading
components $L_j$ with odd $j$,  and let $L^-$ be their sum.
Similarly, let $L_{\rm even }$ denote the set of all `even'
components $L_i$ with even $i$, and let $L^+$ be their sum.
 We also abuse this notation by letting $L_{\rm odd}$ and  $L_{\rm even }$ denote the unions of the corresponding
components.  We use similar notation for any homogeneous additive subgroup $X$ and its components.
\end{notat}

 For $L_i,L_j\in L_{\rm odd}$ and $L_k,L_l\in L_{\rm even}$ we clearly have $[L_i,L_j]\in L_{\rm even}$,  $[L_i,L_k]\in L_{\rm odd}$, and $[L_k,L_l]\in L_{\rm even}$.  Therefore $L^+$ is a subring of $L$, while $L^-$ is not. Note also that the subring generated by  $L_{\rm odd}$ is an ideal of $L$.

The next theorem is essentially a reformulation of  Shumyatsky's theorem for Lie rings
\cite{shu01}.

\begin{theorem}[{\cite[Proposition~2.6]{shu01}}]\label{t-shu} Suppose that  $L= L_0 \oplus L_1 \oplus \dots
\oplus L_{2^n-1}$ is a $({\Bbb Z}/2^n{\Bbb Z})$-graded Lie ring such that  the subring $L^+$ is nilpotent of class
$c$ and $L_0=0$. Then the subring generated by $L^-$ is nilpotent
of $(n,c)$-bounded nilpotency class $f(n,c)$.
\end{theorem}

Since $\langle L^-\rangle ={}_{\rm id}\langle L^-\rangle$,  under the hypotheses of Theorem~\ref{t-shu} the Lie ring $L$ satisfies
$$
\gamma_{f(n,c)+1}(\gamma_{c+1}(L))=0.
$$

For dealing with a Lie ring whose additive group
is a finite $2$-group we  need the following `combinatorial' corollary  of Theorem~\ref{t-shu}. Slightly abusing   notation, we use the same symbols $\gamma _i$ (as those denoting terms of the lower central series) to denote Lie polynomials that are simple multilinear commutators and  write
\begin{align*}
(\gamma_{i}\circ \gamma_{j})(x_1,x_2,\ldots, x_{ij})&=\gamma_{i} \big(\gamma_{j}(x_1, \ldots, x_{j}),\dots , \gamma_{j} (x_{(i-1)j+1},\ldots,
x_{ij})\big)\\ & =\big[[x_1, x_2,\ldots, x_{j}],
[x_{j+1},\ldots, x_{2j}],\ldots, [x_{(i-1)j+1},\ldots,
x_{ij}]\big].
\end{align*}

\begin{corollary}\label{cor-shu}
Let $n, c$ be positive integers, and $f=f(n,c)$ the value of the function given by
Theorem~$\ref{t-shu}$. For $r=(c+1)(f+1)$, the following holds. If
we arbitrarily and formally assign lower indices $i_1, i_2, \ldots
, i_r$ to elements $y_{i_1}, y_{i_2}, \dots, y_{i_r}$ of an
arbitrary Lie ring, then the commutator
$(\gamma_{f+1}\circ \gamma_{c+1})(y_{i_1},y_{i_2},\ldots,
y_{i_r})$ can be represented as a linear combination of
commutators in the same elements $y_{i_1}, y_{i_2}, \dots,
y_{i_r}$ each of which contains either a subcommutator with zero
modulo $2^n$ sum of indices or a subcommutator of weight $c+1$ of the form $[g
_{2u_1}, g_{2u_2},\dots ,g_{2u_{c+1}} \,]$ with even indices, where every element $ g_{2j}$ is a commutator in
$y_{i_1}, y_{i_2}, \dots, y_{i_r}$ such that the sum of indices of
all the elements involved in $g_{2j}$ is congruent to $2j$
modulo~$2^n$.
\end{corollary}

\begin {proof}
  Let $M$ be a free Lie ring freely generated by $x_{i_1},
x_{i_2}, \dots, x_{i_{r}}$. For each
$i=0,\,1,\,\dots ,2^n-1$, let  $M_i$ be the additive subgroup of $M$
generated by all commutators in the
generators $x_{i_j}$ with the sum of indices congruent to $i$
modulo $2^n$. Then, obviously, $M=M_0\oplus M_1\oplus \cdots
\oplus M_{2^n-1}$ and $ [M_i,M_j]\subseteq M_{i+j\,({\rm mod\,
2^n)}}$, so  this is a $({\Bbb Z} /2^n{\Bbb Z})$-grading.
 By Theorem~\ref{t-shu} we obtain
 $$
 (\gamma_{f+1}\circ \gamma_{c+1})(x_{i_1},x_{i_2},\ldots,
x_{i_r}) \in {}_{{\rm
id}}\!\left<M_0\right>+
\gamma _{c+1}\big(M_0+M_2+\dots +M_{2^{n}-2}\big).
$$
By the definition of the $M_i$ this inclusion
is equivalent to the required equality for $y_{i_j}=x_{i_j}$. Since the elements
$x_{i_1}, x_{i_2}, \dots x_{i_r}$ freely generate the Lie ring
$M$, the same equality holds for any elements $y_{i_j}$ in any Lie ring.
\end{proof}

The following theorem was proved by P.~Hall \cite{hall} for
groups; the assertion for Lie rings is proved by  essentially the
same (even simpler) arguments. The bound for the nilpotency class
was later improved by other authors, up to the best possible bound
in \cite{stew}.

\begin{theorem}[{P.~Hall \cite{hall}}]\label{t-hall}
If a Lie ring $L$ has a nilpotent ideal $K$ of nilpotency
class $k$ such that the quotient $L/[K,K]$ is nilpotent of class
$l$, then $L$ is nilpotent of $(k,l)$-bounded class.
\end{theorem}

The following lemmas  are well-known properties of fixed-point
subgroups. As a rule, the induced automorphism of a quotient group
by an invariant  normal subgroup is denoted by the same letter.

\begin{lemma}\label{l-fp} Let $\alpha$ be an automorphism of a finite group
$G$, and $N$ a normal $\alpha$-invariant subgroup of $G$.

{\rm (a)} Then $|C_{G/N}(\alpha )|\leqslant |C_{G}(\alpha )|$.

{\rm (b)} If in addition $(|N|, |\alpha |)=1$, then $C_{G/N}(\alpha )=
C_{G}(\alpha )N/N$. \end{lemma}

The following lemma follows from the consideration of the Jordan normal form of the automorphism regarded as a linear transformation of invariant elementary abelian sections.

\begin{lemma}\label{l-fpp2} Let $p$ be a prime number and suppose
that a
finite abelian group~$A$ of exponent~$p^a$  admits an automorphism of order~$p^k$ with exactly $p^b$ fixed points. Then $|A|\leqslant p^{abp^k}$.
\end{lemma}

 Recall that a  \emph{multilinear} (or \emph{outer})
\emph{commutator} is any commutator $\varkappa $ of weight $w$  in $w$ distinct group  variables; in other words, $\varkappa $ is obtained by nesting commutators, but using always
different variables. Laws $\varkappa =1$ for multilinear commutators $\varkappa $ define many popular soluble group varieties, including those of nilpotent groups of given class, and of soluble groups of given derived length. The following  Khukhro--Makarenko theorem \cite{khu-mak-char} greatly facilitates working with subgroups of finite index satisfying a multilinear commutator law. In the special case of nilpotency laws this result was obtained by Bruno and Napolitani \cite{brna}. (Further generalizatons and improvements of this theorem were obtained in \cite{kh-kl-ma-me,kl-me,kl-mi}.)

\begin{theorem}[{\cite[Theorem~1]{khu-mak-char}}]\label{t-char}
If a group $G$ has a subgroup $H$ of finite
index $k$ satisfying the law $\varkappa (H)=1$, where
$\varkappa$ is a multilinear commutator of weight $w$, then $G$
also has a characteristic subgroup $C$ of finite $(k,w)$-bounded index
satisfying the same law $\varkappa (C)=1$.
\end{theorem}

\section{Lie rings}\label{s-rings}

The bulk of the proof of Theorem~\ref{t2} is about $({\Bbb Z}/2^n{\Bbb Z})$-graded Lie rings considered in the following proposition.

\begin{proposition}\label{p1} Suppose that  $L= L_0 \oplus L_1 \oplus \dots
\oplus L_{2^n-1}$ is a $({\Bbb Z}/2^n{\Bbb Z})$-graded Lie ring.
    Suppose that the  subring $L^+$ is nilpotent of class $c$, while the component $L_0$ is finite of order $m$. Then $L$ contains
    a homogeneous nilpotent ideal $M$ of $(n,c)$-bounded nilpotency class such that $M\cap L^-$ has $(m,n)$-bounded index in the additive group $L^-$.
\end{proposition}

By Theorem~\ref{t3} the Lie ring $L$ contains a soluble
homogeneous ideal of $n$-bounded derived length and of
$(m,n)$-bounded index in $L$. Therefore
Proposition~\ref{p1} will be proved if we prove the following proposition, taking
advantage of induction on the derived length.

\begin{proposition}\label{p2}
Suppose that under the hypotheses of Proposition~\ref{p1} the Lie ring $L$ has a soluble homogeneous ideal
$A$ of derived length $d$ such that $A\cap L^-$ has index $l$ in the additive group $L^-$.
Then
 $L$ contains a homogeneous nilpotent ideal $B$ of $(d,n,c)$-bounded  nilpotency class such that $B\cap L^-$
 has $(d,l,m,n)$-bounded index in $L^-$.
\end{proposition}

\begin{proof} The sought-for ideal $B$ is constructed by using  certain
additive subgroups $L_j(t)\leqslant L_j$ of the components $L_{j}$,
so-called graded centralizers of levels $t=1,2,3$. We also use induction of $d$. Clearly, if $d=1$, then $A$ is abelian and we can put $B=A$. So we assume that $A$ is not
abelian.

Let $R=A^{(d-2)}$ be the penultimate (metabelian) term of the derived series
of $A$. First we construct graded centralizers
$R_j(1)\leqslant R_j$ in $R$, which are additive subgroups of $(m,n)$-bounded
index in the grading  components $R_j=R\cap L_j$, and fix certain
elements  called $r$-representatives,  whose total
number is $(m,n)$-bounded.

\begin{definition*} The \emph{pattern\/} of a homo\-geneous
commutator  is its bracket structure together with the arrangement
of the indices under the Index Convention. The \emph{weight\/} of a
pattern is the weight of the commutator. The commutator is then
called the \emph{value\/} of its pattern on the given elements. For
example, $[a_2,[b_1,b_1]]$ and $[x_2,[z_1,y_1]]$ are values of the
same pattern of weight~3.
\end{definition*}

\paragraph{\bf \boldmath Definition of representatives
in $R$.}  For every $i\ne 0$ and for every pair $({\bf P},a_0)$
 consisting of the pattern ${\bf P}$ of a
simple commutator of weight $2^n$ with one and the same index
$i\ne 0$
 (repeated $2^n$ times) and a commutator $a_0\in R_0$ that is the
value of this pattern on homo\-geneous elements of~$R_{i}$ we fix
one such  representation. (The same element $a_0\in R_0$ may appear
in different pairs if it is equal to values of different patterns;
the same pattern may appear in different pairs if different
commutators are the values of this pattern.) The elements of
$R_{j}$,\, $j\ne 0$, involved in these fixed representations are
called \emph{$r$-representatives} and
 denoted by $r_j(0)\in R_j$ under the Index
Convention: recall that the same symbol can denote different
elements. Thus, the commutator $a_0$ mentioned above is equal to
$[\underbrace{r_i(0),\ldots ,r_{i}(0)}_{2^n}]$.
 Since the total number of patterns ${\bf P}$ under
consideration is equal to $2^n-1$ and the number of elements in
$R_0$ is at most $m$, the number of $r$-representatives
 is $(m,n)$-bounded.
\medskip

The definition of $r$-representatives implies the following.

\begin{lemma} \label{l-freezer}
 Every simple homo\-geneous commutator in elements of $R$ of length $2^n$ with one and
the same  index $i\ne 0$ repeated $2^n$ times can be represented
as a commutator of the same pattern in
$r$-representatives. \end{lemma}

Before defining graded centralizers, we introduce the
following homomorphisms.

\begin{definition} \label{d-th}
Let $\vec z=(z_{i_1},\ldots ,z_{i_k})$ be an ordered tuple of
elements $z_{i_s}\in L_{i_s}$, \, $i_s\ne 0$, such that
$i_1+\cdots + i_k\not\equiv 0\, (\mbox{mod}\, 2^n)$. We put $j=-
i_1-\cdots - i_k\,  (\mbox{mod}\,  2^n)$ and define
 the mapping
 $$
 \vartheta _{\vec z}:\, \, y_j\rightarrow
[y_j,\, z_{i_1},\,\ldots ,\, z_{i_k}].\label{e1}
$$
By linearity this is a homomorphism of the additive group $L_j$
into $L_0$ 
by
the choice of~$j$.
 Since $|L_0|\leqslant m$, we have $|L_j:\mbox{Ker}\,
\vartheta _{\vec z}|\leqslant m$. Clearly, we also have
$|R_j:\mbox{Ker}\, \eta _{\vec z}|\leqslant m$ for the restriction
$\eta _{\vec z}$ of $\vartheta _{\vec z}$ to $R_j$.
\end{definition}

\paragraph{\bf \boldmath Definition of graded centralizers in $R$.}
 We define the \emph{graded
centralizers in $R$}
by setting for each $i\ne 0$
$$
R_i(1)=\bigcap_{\vec r}\,\mbox{Ker}\,\eta _{\vec r},
$$
where $ \eta _{\vec r}$ is defined in Definition~\ref{d-th} with
$\vec r=\left(r_{i}(0), \ldots , r_{i}(0)\right) $ running over
all possible ordered tuples of length  $2^n-1$ consisting of
(possibly different) $r$-representatives with the same index
$i$.  Elements of $R_i(1)$ for $i=1,\dots ,2^n-1$ are also
called \emph{centralizers in $R$} for short and are denoted by
$r_i(1)$ (under the Index Convention). The number of
$r$-representatives is $(m,n)$-bounded and $|R_i:\mbox{Ker}\, \eta
_{\vec r}|\leqslant m$ for all $\vec r$. Hence this is an intersection
of $ (m,n) $-boundedly many subgroups of $m$-bounded index
in~$R_i$ and therefore $R_i(1)$
 also has $(m,n)$-bounded index in the additive
group~$R_{i}$.
\medskip

By construction, we have the following centralizer property:
\begin{equation}\label{e-cr} \big[ r_i(1), \underbrace{r_{i}(0),\,\ldots ,\,
r_{i}(0)}_{2^n-1} \big]=0 \end{equation} 
for any centralizer  $r_i(1)\in
R_i(1)$ in $R$  and any $r$-representatives $r_i(0)$ with the same
index $i\ne 0$. (Here, as always under the Index Convention,
 the elements $r_{i}(0)$ can be different.)

We also need to introduce another set of representatives in
$R$.\medskip

\paragraph{\bf \boldmath Coset representatives in $R$}  For each $j\ne 0$ we fix an arbitrary system of coset representatives of the subgroup $R_j(1)$ in the additive
group~$R_j$. These elements are denoted by $ q_j\in  R_j$ (under
the Index Convention) and called \emph{coset repre\-sen\-ta\-tives
in $R$}. The total number of coset repre\-sen\-ta\-tives is
$(m,n)$-bounded, since the indices $|R_j:R_j(1)|$ are
$(m,n)$-bounded for all $j\ne 0$ by construction.
\medskip

Our next construction is of representatives and graded
centralizers in the whole ring  $L$.
\smallskip

\paragraph{\bf \boldmath Definition of level $1$ for $L$.}
 We define the \emph{graded
centralizers of level $1$} in $L$  by setting for each $i\ne 0$
$$
L_i(1)=\bigcap_{\vec z}\,\mbox{Ker}\,\vartheta _{\vec z},
$$
where the $ \vartheta _{\vec z}$ are defined in
Definition~\ref{d-th} and $\vec z=( q_{j}, \ldots , q_{j}) $
runs over all possible ordered tuples of length  $k\leqslant 2^n-1$
consisting of coset representatives in $R$ with the same index
$j\ne 0$ such that
$$
i+ kj\equiv 0\, (\mbox{mod}\, 2^n).
$$
(Under the Index Convention the tuple $\vec
z=( q_{i}, \ldots , q_{i}) $ may consist of different coset
representatives $q_i$.)  Elements of the $L_i(1)$ are also
called \emph{centralizers of level $1$ in $L$} and are denoted by
$y_i(1)$ (under the Index Convention). The number of coset
representatives in $R$ is $(m,n)$-bounded and $|L_i:\mbox{Ker}\,
\vartheta _{\vec z}|\leqslant m$ for all $\vec z$. Hence this is an
intersection of $ (m,n) $-boundedly many subgroups of $m$-bounded
index in~$L_i$ and therefore $L_i(1)$
 also has $(m,n)$-bounded index in the additive
group~$L_{i}$.

 For each
$j\ne 0$ we also fix an arbitrary system of coset representatives of
the subgroup $L_j(1)$ in the additive group $L_j$. These elements
are denoted by $ b_j(1)$ (under the Index Convention) and called
\emph{coset repre\-sen\-ta\-tives of level 1 in $L$}. The total
number of coset repre\-sen\-ta\-tives is $(m,n)$-bounded, since
the indices $|L_j:L_j(1)|$ are $(m,n)$-bounded for all
$j=1,\,2,\ldots ,2^n-1$.

\medskip

\paragraph{\bf \boldmath Definition of level $2$ in $L$.}
 We define the \emph{graded centralizers of
level $2$ in $L$} by setting for each $j\ne 0$
$$
L_j(2)=L_j(1)\cap \bigcap_{\vec z}\,\mbox{Ker}\,\vartheta _{\vec
z},
$$
where the $ \vartheta _{\vec z}$ are defined in
Definition~\ref{d-th} and $\vec z=\left( b_{i_1}(1),\, \ldots ,\,
b_{i_k}(1)\right) $ runs over all possible ordered tuples of all
lengths  $k\leqslant 2^{3n+1}$ consisting of coset representatives
of level 1 in $L$ such that
$$
j+ i_1+\cdots + i_k\equiv 0\, (\mbox{mod}\, 2^n).
$$
Elements of the $L_j(2)$ are  called \emph{centralizers of level $
2$} and are denoted by $y_j(2)$ (under the Index Convention). The
number of coset representatives of level $1$ in $L$ is
$(m,n)$-bounded and $|L_j:\mbox{Ker}\, \vartheta _{\vec z}|\leqslant m$
for all $\vec z$. Hence this is an intersection of $ (m,n)
$-boundedly many subgroups of $m$-bounded index in~$L_j$ and
therefore $L_j(2)$
 also has $(m,n)$-bounded index in the additive
group~$L_{j}$.

For each $j\ne 0$, we now fix an
arbitrary system of coset representatives of the subgroup $L_j(2)$
in the additive group $L_j$. These elements are
denoted by $ b_j(2)$ (under the Index Convention) and called {\it
coset repre\-sen\-ta\-tives of level~$2$ in $L$}. The total number
of coset repre\-sen\-ta\-tives of level $2$ is $(m,n)$-bounded,
since the indices $|L_j:L_j(2)|$ are $(m,n)$-bounded for all
$j=1,\,2,\ldots ,2^n-1$.

\medskip

\paragraph{\bf \boldmath Definition of level $3$ in $L$.}
We define the \emph{graded centralizers of level $3$ in $L$} by
setting for each $j\ne 0$
$$
L_j(3)=L_j(2)\cap \bigcap_{\vec z}\,\mbox{Ker}\,\vartheta _{\vec
z},
$$
where the $ \vartheta _{\vec z}$ are defined in
Definition~\ref{d-th} and $\vec z=\left( b_{i_1}(2),\, \ldots ,\,
b_{i_k}(2)\right) $ runs over all possible ordered tuples of all
lengths  $k\leqslant 2^{3n+1}$  consisting of coset representatives of
level 2 in $L$ such that
$$
j+ i_1+\cdots + i_k\equiv 0\, (\mbox{mod}\, 2^n).
$$
Elements of the $L_j(3)$ are  called \emph{centralizers of level $
3$} and are denoted by $y_j(3)$ (under the Index Convention). The
number of coset representatives of level $2$ in $L$ is
$(m,n)$-bounded and $|L_j:\mbox{Ker}\, \vartheta _{\vec z}|\leqslant m$
for all $\vec z$. Hence this is an intersection of $ (m,n)
$-boundedly many subgroups of $m$-bounded index in~$L_j$ and
therefore $L_j(3)$
 also has $(m,n)$-bounded index in the additive
group~$L_{j}$.

The construction of centralizers and coset representatives of
levels $\leqslant 3$ in $L$ is complete.
\medskip

Note that by construction we have
\begin{equation}\label{e-incl}
L_j(k+1)\leqslant L_j(k)
\end{equation}
for all $j$ and $k$.

The definition of centralizers $y_v(1)$ of level 1 implies the
following centralizer property with respect to coset
representatives in $R$:
\begin{equation}\label{e-clr}
[ y_i(1),
\underbrace{q_{j},\ldots , q_{j}}_{k}\,]=0,
\end{equation}
 for any $k\leqslant
2^{n}-1$ for any (possibly different) coset representatives
$q_j$ in $R$ with the same index $j$ such that $i+kj\equiv
0\,({\rm mod}\,2^n)$.

The definitions of centralizers $y_v(t)$ of levels $t=2,3$ imply
the following centralizer property with respect to coset
representatives in $L$ of the preceding level:
\begin{equation}\label{e-cll}
[y_j(t), b_{i_1}(t-1),\ldots , b_{i_k}(t-1)]=0,
\end{equation}
for any $k\leqslant
2^{3n+1}$ for any coset representatives in $L$ of level $t-1$ such
that $j+i_1 +\dots +i_k\equiv 0\,({\rm mod}\,2^n)$.

The following two lemmas are similar to
\cite[Lemma~3]{khu-mak04} and a special case of
\cite[ Lemma~9]{khu-mak04}, but we have to reproduce the proofs, since the
definitions of representatives and graded centralizers here are somewhat different.

\begin{lemma} \label{l-3}
 Any commutator of the form $[a_{-i},\,y_i(k)]$,
 where $y_i(k)$ is a centralizer of level $k=2,3$, is equal to a
commutator of the form $[y_{-i}(k-1),\,y_i(k)]$, where
$y_{-i}(k-1)$ is a centralizer of level $k-1$. \end{lemma}

\begin{proof} We have $a_{-i}=b_{-i}(k-1)+y_{-i}(k-1)$ for some coset representative $b_{-i}(k-1)$ and a centralizer
$y_{-i}(k-1)$ of level $k-1$. Then
$$
[a_{-i},\,y_i(k)]=[b_{-i}(k-1),\,y_i(k)]+[y_{-i}(k-1),\,y_i(k)]=[y_{-i}(k-1),\,y_i(k)],
$$
since $[b_{-i}(k-1),\,y_i(k)]=0$ by~\eqref{e-cll}. \end{proof}

\begin{lemma}\label{l-9} For any $j\ne 0$ any commutator \begin{equation}\label{e-91}
[y_j(3),\,a_{k_1},\,a_{k_2},\ldots ,a_{k_s}]\in L_0 \end{equation}
 (under the
Index Convention), of any length, for any indices
$k_i\in\{0,1,\dots ,2^{n-1}-1\}$  such that  $ j+k_1+\cdots
+k_s\equiv 0\;({\rm mod\;}2^n)$ is equal to a linear combination
of elements of the form $[y_{-k}(1),y_k(1)]$  for various $ k\ne
0$. \end{lemma}

\begin{proof} We use induction on $s$. If $s=0$ there is nothing to prove,
since $ j\ne 0$.
 If $s=1$, this follows from Lemma~\ref{l-3}, by which  $[y_j(3),\,a_{-j}]=
[y_j(3),\, y_{-j}(2)]$, and from the inclusions \eqref{e-incl}.

For $s>1$ by the Jacobi identity we can permute the elements
$a_{k_u}$ in the  commutator \eqref{e-91} modulo
$$
\sum \limits_{t=1}^{s-1}\;\sum \limits_{j+i_1+\cdots +i_t\equiv
0\;({\rm mod\;}2^n)}[L_j(3),L_{i_1},\ldots ,L_{i_{t}}].
$$
By the induction hypothesis all elements in this sum can be
expressed in the required form. Therefore we may freely permute
the $a_{k_u}$ in \eqref{e-91} in order
 to express our commutator  in the
required form.

We express every element $a_{k_u}$ in \eqref{e-91} with non-zero
index $k_u\ne 0$ in the form $b_{k_u}(2)+y_{k_u}(2)$ and
substitute all these expressions into the commutator \eqref{e-91}.
We obtain a linear combination of commutators
$$
[y_j(3),\,z_{k_1},\,z_{k_2},\ldots ,z_{k_s}],
$$
where the $z_{k_u}$ are either $b_{k_u}(2)$, or $y_{k_u}(2)$, or
$a_0$, and $ j+k_1+\cdots +k_s\equiv 0\;({\rm mod\;}2^n)$. If among the $z_{k_u}$ there   is
at least one $y_{k_u}(2)$, then we transfer it to the right end of
the commutator, denote by $a_{-k_u}$ the preceding initial
segment, and apply Lemma~\ref{l-3}: $ [a_{-k_u},
y_{k_u}(2)]=[y_{-k_u}(1), y_{k_u}(2)]$, which is of required form
by the inclusions \eqref{e-incl}.

Hence it remains to consider the case of a commutator
\begin{equation}\label{e-9}
[y_j(3),\,z_{k_1},\,z_{k_2},\ldots ,z_{k_s}],
\end{equation}
where all the $z_{k_u}$ are either $b_{k_u}(2)$ with $k_u\ne 0$ or
$a_0$, and $ j+k_1+\cdots +k_s\equiv 0\;({\rm mod\;}2^n)$.
(Note that the $z_i$ cannot all be $a_0$, since $ j\ne 0$.) We now prove that such a
commutator is actually equal to~0. We do this by showing that
some of the entries $b_{k_u}(2)$ can be placed at the beginning
after $y_j(3)$  producing  an initial segment of bounded weight
with zero sum of indices modulo $2^n$, which is equal to $0$ by
\eqref{e-cll}.

For each index $u\ne 0$ that occurs less than $2^{2n}$ times  we
 transfer all the $b_u(2)$ (if any) to the left to place them right after
$y_j(3)$ (in any order). Let $\hat y_{t}\in L_t$ denote the
initial segment of length $<2^{3n}$ formed in this way. If there
are no other indices, that is, indices  $k\ne 0$ for which there
are at least $2^{2n}$ elements $ b_{k}(2)$ in the commutator, then
the only elements outside $\hat y_{t}$ are $a_0$ and we must have
$t=0$, since the original sum of indices was $0$ modulo~$2^n$.
Then  $\hat y_{t}=0$  by \eqref{e-cll}  and the proof is
complete.

Thus, we can assume that there are non-zero indices $v_1,\ldots ,
v_r$, where $1\leqslant r\leqslant 2^n-1$, such that for each $v_i$ there
are at least $2^{2n}$ elements $ b_{v_i}(2)$ in the commutator \eqref{e-9}.
Let $v={\rm gcd}(v_1,\ldots ,v_r)$ be the greatest common divisor
of $v_1, \ldots ,v_r$. Since the sum of all indices is $ 0$ modulo
$2^n$, the number ${\rm gcd}(v,2^n)$ must divide~$t$.
 By the Chinese
remainder theorem there exist integers $u_i$ such that  $ v= u
_1v_1+\cdots +u_rv_r$. Replacing the $u_i$ by their residues
modulo $2^n$ and changing notation we have $ v=u _1v_1+\cdots
+u_rv_r+u2^n$, where $ u_i\in \{ 0,\,1,\ldots ,2^n-1\}$  for all $
i$ and $u$ is an integer.
Since ${\rm gcd}(v,2^n)$ divides~$t$, there is $w\in \{ 0, 1,\dots  , 2^n - 1\}$ such that  $t + wv \equiv 0 \;({\rm mod}\; 2^n)$. Substituting the expression for $v$ we obtain
\begin{equation}\label{e-congr}
0\equiv  t + wv \equiv t + wu_1v_1 + \dots  + wu_rv_r \;({\rm mod}\; 2^n).
\end{equation}

 We now arrange an initial segment of the commutator by placing  after
$\hat y_{t}$  exactly $wu_1$ elements $b_{v_1}(2)$, then exactly
$wu_2$ elements $ b_{v_2}(2)$, and so on, up to exactly $wu_r$
elements $b_{v_r}(2)$. This is possible because $wu_i\leqslant 2^{2n}$
for each $i$, and there are at least $2^{2n}$ elements  $
b_{v_i}(2)$  outside $\hat y_{t}$.
The resulting  initial segment
has zero sum of indices modulo $2^n$ by \eqref{e-congr} and has length
 $\leqslant  2^{3n}+
 2^{3n}$. Hence
it is equal to $0$ by \eqref{e-cll}. \end{proof}

We now proceed with the proof of Proposition~\ref{p2}. Consider
the ideal $I={}_{\rm id}\big\langle L_{\rm odd}(3)\big\rangle \cap
A$.
 Clearly, $I\cap L^-$ has $(l,n,m)$-bounded index in $L^-$,
since $A\cap L^-$ has index $l$ in $L^-$, and each component in
$L_{\rm odd}(3)$ has $(n,m)$-bounded index in the corresponding
component in $L_{\rm odd}$. Let $S=I^{(d-2)}$  be the $(d-2)$-nd
term of the derived series of $I$. Note that, clearly, $S\leqslant
R=A^{(d-2)}$.

\begin{lemma}\label{l-r1} The ideal $S$ is nilpotent of $(n,c)$-bounded
class. \end{lemma}

  \begin{proof}
Recall that we write $S^-=\sum S_{\rm odd}$ and $S^+=\sum S_{\rm
even}$, and similarly for $[S,S]$. We represent $S$ as the sum of
two ideals $S= J_1+J_2$, where $J_1=[S,S]+S^-$ and  $J_2=[S,
S]+S^+$.  Since $S$ is metabelian, $[S, S]^-$ is an ideal of
$J_2$. By hypothesis, $\gamma_{c+1}(J_2)\leqslant [S, S]^-$.

We  claim that  $J_1$ is nilpotent of $n$-bounded class. For
that, we need to show that any simple commutator
$$
[a_{i_1},a_{i_2},a_{i_3},\dots ]
$$
of large enough $n$-bounded length in homogeneous elements of
$J_1$ is equal to~0. Since $S$ is metabelian, we can assume that
all the entries starting from the third one are from $S^-$, so the
commutator is a linear combination of commutators of  the form
\begin{equation}\label{e-j1}
 [[a_{i_1},a_{i_2}],a_{\rm odd}, a_{\rm odd}, \dots ],
\qquad a_{i_j}\in S_{i_j},\quad a_{\rm odd}\in S_{\rm odd},
\end{equation}
and all
the entries $a_{\rm odd}$ can be freely permuted without changing the
commutator. When the length is large enough, we can rearrange
these entries in such a way that there will be an initial segment
in $[S,S]_0$ of $n$-bounded length, which we denote by $w_0$. With
large enough length of \eqref{e-j1} there will remain at least
$2^{n-1}(2^{n+1}+2^n-4)+1$  elements $a_{\rm odd}$ outside the initial
segment $w_0$, and therefore at least $2^{n+1}+2^n-3$ of them with
the same (odd) index, say, $j$. These entries $a_j$ can be
moved to be placed at the beginning after $w_0$. Therefore  it suffices to prove that
the commutator 
\begin{equation}\label{e-j2} [w_0,\underbrace{a_{j}, a_{j},
\dots ,a_j }_{2^{n+1}+2^n-3}\,], \qquad \text{where }2\nmid j, \end{equation}
is equal to zero.

Since $S\leqslant R$ we can represent all the  entries $a_{j}$ in
\eqref{e-j2} in the form
 $a_{j}=r_{j}(1)+ q_{j}$, where the $r_{j}(1)$ are centralizers in $R$, and the $q_{j}$  are coset representatives in $R$. (Note that these elements may no longer be in $I$). After expanding all brackets, we obtain a linear combination of commutators
  \begin{equation}\label{e-j3}
[w_0,\underbrace{z_{j}, z_{j}, \dots ,z_j}_{2^{n+1}+2^n-3}\,],
\qquad \text{where }2\nmid j, \end{equation}
   and
   each $z_{j}$ is either $r_{j}(1)$ or $q_{j}$. In each of such
commutators there are  either
 at least $2^n+1$ entries $r_{j}(1)$, or
 at least $2^{n+1}-3$ entries $q_{j}$.

By permuting the entries $r_{j}(1)$ and  $q_{j}$ (we can freely
permute  these elements, since  $R$ is metabelian and  $w_0 \in [R, R]$, $r_{j}(1), q_{j}\in R$), we obtain from \eqref{e-j3} either a
commutator with an initial segment 
\begin{equation}\label{e-j4}
[w_0,\underbrace{r_{j}(1), r_{j}(1), \dots
,r_{j}(1)}_{2^{n}+1}\,], \qquad \text{where }2\nmid j, \end{equation}
 or a
commutator with an initial segment 
\begin{equation}\label{e-j5}
[w_0,\underbrace{q_{j}, q_{j}, \dots ,q_{j}}_{2^{n+1}-3}\,],
\qquad \text{where }2\nmid j. \end{equation} 
Thus, it suffices to show that
both commutators \eqref{e-j4} and \eqref{e-j5} are equal to~0.

In the commutator \eqref{e-j4} we regard the initial segment of
the first two entries  $ a_j= [ w_0, r_{j}(1)]$ simply as an
element of the ideal $R$ which belongs to $R_j$. By
Lemma~\ref{l-freezer}  we can represent  the initial segment
$$
[a_j , \underbrace{r_{j}(1), r_{j}(1), \dots
,r_{j}(1)}_{2^{n}-1}\,]\in L_0
$$
in terms of $r$-representatives of level 0, so that the commutator
\eqref{e-j4} becomes equal to
$$
[ [\underbrace{r_j(0), \dots ,r_j(0)}_{2^{n}}\,], r_{j}(1)].
$$
This commutator in turn is equal to a linear combination of
commutators of the form
$$
[ r_{j}(1),  \underbrace{r_j(0), \dots ,r_j(0)}_{2^{n}}\,],
$$
in each of which  the initial segment of length $2^n$ is equal to
0 by \eqref{e-cr}.

We now consider the commutator \eqref{e-j5}. Its initial segment
$w_0$, being in $I$,  also belongs to ${}_{\rm id}\langle
L_{\rm odd}(3)\rangle$ and therefore is a linear combination of
elements of the form 
\begin{equation} \label{e-long} [l_{\rm odd}(3), u_{i_1},
u_{i_2}, \dots , u_{i_s}]
 \end{equation}
with zero sum of indices modulo $2^n$, where $u_{i_k}\in L_{i_k}$ are arbitrarily homogeneous elements, in any number. By
Lemma~\ref{l-9} an element of the form \eqref{e-long} can be
represented as a linear combination of elements of the form
$[y_{-k}(1), y_{k}(1)]$, for various, not necessarily odd, $k$
(and these elements are not necessarily contained in $I$).
Therefore  the commutator  \eqref{e-j5} is a  linear combination
of commutators of the form 
\begin{equation}\label{e-s} \big[ [y_{-k}(1),
y_{k}(1)], \underbrace{q_j, q_{j}, \dots
,q_{j}}_{2^{n+1}-3}\,\big]. \end{equation} 
By the Jacobi identity, the
commutator \eqref{e-s} is equal to a linear combination of
commutators of the form
$$
\big[ [y_{-k}(1),  \underbrace{q_j, q_{j}, \dots ,q_{j}}_{k}\, ]
,\, [y_{k}(1), \underbrace{q_j, q_{j}, \dots
,q_{j}}_{2^{n+1}-3-k}\,]\big].
$$
In such a commutator, one of the two subcommutators contains a
subcommutator of the form
$$
 [y_{\pm k}(1),  \underbrace{q_j, q_{j}, \dots ,q_{j}}_{2^n-1}\, ].
$$
Since $j$ is odd, there is an initial segment in $L_0$, which is
equal to 0 by \eqref{e-clr}.

 Thus, we have proved that the ideal $J_1$ of $S$ is nilpotent of $n$-bounded class $c_1$.

As a  result,
\begin{align*}
\gamma_{c+1+c_1+1}(S)&=\gamma_{c+1+c_1+1}(J_1+J_2)\\
                                       &\leqslant  \gamma_{c_1+1}(J_1)+\gamma_{c+1}(J_2)\\
                                       &\leqslant [S, S]^-.
\end{align*}
But $S$ is an ideal of $L$, and hence $\gamma_{c+1+c_1+1}(S)$ is
also an ideal of $L$. Since $[[S, S]^-, L_{\rm odd}]\leqslant L^+$, the
inclusion $\gamma_{c+1+c_1+1}(S)\leqslant [S, S]^-$ implies that
$[\gamma_{c+c_1+2}(S), L^-]=0$. This means that $L^-$ is contained
in the centralizer of the ideal $\gamma_{c+c_1+2}(S)$, and then
$[\gamma_{c+c_1+2}(S), {}_{\rm id}\langle L^-\rangle ]=0$. In
particular, $[\gamma_{c+c_1+2}(S), I]=0$, and therefore
$\gamma_{c+c_1+3}(S)=0$.

The lemma is proved. \end{proof}

We now complete the proof of Proposition~\ref{p2}. The quotient
$L/[S,S]$ contains the homogeneous  ideal $I/[S,S]$ of derived
length at most $d-1$ and its intersection with the image of $L^-$
has  $(l,n,m)$-bounded index $t$ in the image of $L^-$. By the
induction hypothesis, there is a homogeneous nilpotent ideal
$J/[S,S]$ of $(d-1,n,c)$-bounded class whose intersection with the
image of $L^-$ has  $(d-1,t,m,n)$-bounded index in the image of
$L^-$. Then the quotient $(S+J)/[S,S]$ is also nilpotent of
$(d-1,n,c)$-bounded class. Since $S$ is nilpotent of
$(n,c)$-bounded class by Lemma~\ref{l-r1},  we obtain that the
full inverse image $B=J+S$ of $(J+S)/[S,S]$ is nilpotent of
$(d,n,c)$-bounded class by Hall's
Theorem~\ref{t-hall}.
This is a required ideal, since its intersection with $L^-$
has $(d,l,n,m)$-bounded index in $L^-$ (recall that $t$ is an
$(l,n,m)$-bounded number). \end{proof}

\begin{proof}[Proof of Proposition~\ref{p1}]
Proposition~\ref{p1} follows from Proposition~\ref{p2} and Theorem~\ref{t3}.
\end{proof}

\begin{proof}[Proof of Theorem~\ref{t2}] Recall that $L$ is a finite Lie
ring admitting an automorphism $\varphi$ of order $2^n$ such that the
fixed-point subring $C_L(\varphi ^{2^{n-1}})$ of the involution $\varphi 
^{2^{n-1}}$ is nilpotent of class $c$, and $m=|C_L(\varphi )|$ is the
number of fixed points of $\varphi$. We wish to prove that $L$ has
homogeneous ideals $M_1\geqslant M_2$ such that $M_1$ has $(m,n)$-bounded index in the
additive group $L$, the quotient  $M_1/M_2$ is nilpotent of class at most $c+1$,   and $M_2$ is nilpotent of
$(n,c)$-bounded class.

First we extend the ground ring by a $2^n$-th root of unity $\omega$
forming $\tilde L=L\otimes _{{\Bbb Z} }{\Bbb Z} [\omega ]$. Then $|C_{\tilde L}(\varphi
)|\leqslant m^{2^n}$ and $C_{\tilde L}(\varphi ^{2^{n-1}})=C_L(\varphi
^{2^{n-1}})\otimes _{{\Bbb Z} }{\Bbb Z} [\omega ]$ is also nilpotent of class $c$.
Therefore  it is  clearly sufficient to prove the theorem for
$\tilde L$, so we assume that $L=\tilde L$ in what follows.

 We begin with the case where $L$ has odd order. Then
 $$
 L=\bigoplus _{i=0}^{2^n-1}L_i\qquad \text{and}\qquad [L_i,L_j]\leqslant L_{i+j\,(\rm mod\,2^n)}
$$
for the analogues of eigenspaces of $\varphi$
$$
L_i=\{x\in  L\mid  x^{\varphi }=\omega ^ix\}, \qquad i=0,1,\dots ,2^n-1.
$$
Thus, this is a $({\Bbb Z} /2^n{\Bbb Z} )$-grading of $L$ and $L$ satisfies
the hypotheses of Proposition~\ref{p1}.  By that proposition, $L$
contains a homogeneous nilpotent ideal $M_2$ of $(n,c)$-bounded
nilpotency class such that $M_2\cap L^-$ has $(m,n)$-bounded index
in the additive group $L^-$. The latter means that in the
inherited grading of the quotient $\bar L=L/M_2$ the order of
$\bar L^-$ is $(m,n)$-bounded.

For any fixed homogeneous element $a_j\in \bar L ^-$ (so that $j$
is odd) and for any even $k$ the map $b_k\to [a_j,b_k]$ from $\bar
L_k$ to $\bar L_{j+k}\in \bar L_{\rm odd}$ is linear with
$(m,n)$-bounded image. Therefore its kernel has $(m,n)$-bounded
index in $\bar L_k\in \bar L_{\rm even}$. As a result, $C_{\bar
L^+}(a_j)$ has $(m,n)$-bounded index in $\bar L^+$. Taking the
intersection over an $(m,n)$-bounded  number of homogeneous
elements generating $\bar L^-$ we obtain that $|\bar L^+:C_{\bar
L^+}(\bar L^-)|  $ is $(m,n)$-bounded and therefore $|\bar L :
C_{\bar L}(\bar L^-)| $ is also $(m,n)$-bounded, since $|\bar
L^-|$ is $(m,n)$-bounded.

Let $K=\langle \bar L^-\rangle$ be the subring of  $\bar L$
generated by $\bar L^-$. Recall that $K$ is an ideal of $\bar L$, and therefore $C_{\bar L}(K)$ is also an ideal of $\bar L$.
The index $|\bar L : C_{\bar L}(K)| $ is also $(m,n)$-bounded
since $C_{\bar L}(K) =C_{\bar L}(\bar L^-)$. We prove that the
ideal $C_{\bar L}(K)$ is  nilpotent of class at most $c+1$. Indeed,
 $\bar L /K=(\bar L^++K)/K$ is nilpotent of class at most $c$
by hypothesis, and $C_{\bar L}(K)\cap K$ is central in $K$. Hence,
$$
[\underbrace{C_{\bar L}(K),\ldots, C_{\bar L}(K)}_{c+1}, C_{\bar
L}(K) ]\subseteq [K, C_{\bar L}(K)]=0
$$
and therefore
 $\bar M_1=C_{\bar L}(K)$ is a nilpotent ideal
of class at most $c+1$. Then its full inverse image $M_1$ and the
aforementioned ideal $M_2$ satisfy the conclusion of
Theorem~\ref{t2}.

We now consider  the case where the additive group of $L$ is a
finite $2$-group. Although we no longer have a direct sum, it is well known (see, for example, \cite[Ch. 10]{hpbl}) that
$$
2^nL\leqslant L_0+L_1+\cdots +L_{2^n-1}\qquad \text{and}\qquad [L_i,L_j]\leqslant L_{i+j\,(\rm mod\,2^n)}.
$$
By Corollary~\ref{cor-shu} applied to the subring $M=L_0+L_1+\cdots
+L_{2^n-1}$ we have
$$
\gamma_{f(n,c)+1}\big(\gamma_{c+1}(M)\big)
\leqslant {}_{{\rm id}}\!\left<L_0\right>.
$$
It follows that
$$
m \gamma_{f(n,c)+1}\big(\gamma_{c+1}(M)\big) \leqslant m\,\,{}_{{\rm
id}}\!\left<L_0\right>=0,
$$
since $mL_0=0$ by Lagrange's theorem.
Hence,
$$
\gamma_{f(n,c)+1}\big(\gamma_{c+1}(mM)\big)=0.
$$
By Lemmas
\ref{l-fp}(a) and \ref{l-fpp2} the index of the additive subgroup
$2^nmL \leqslant mM$ in $L$ is $(n,m)$-bounded, and hence  $M_1=2^nmL$ and $M_2=\gamma_{c+1}(M_1)$ are
the required ideals.

In the  case of an arbitrary  finite Lie ring,  $L$  is a direct
sum of two ideals $L=T_2\oplus T_{2'}$, where the additive group $T_2$ is the Sylow $2$-subgroup of $L$, and $T_{2'}$ is the Hall $2'$-subgroup of $L$. As shown above,
$T_2$ and $T_{2'}$ contain ideals $I_1$ and $I_2$, respectively,  of
$(m,n)$-bounded indices such that
$$
\gamma_{g(n,c)}\big(\gamma_{c+2}(I_k)\big)=0, \qquad k=1,2
$$
for some $(n,c)$-bounded number $g(n,c)$. Since $[T_2, T_{2'}]=0$, it follows that $I_1$ and $I_2$ are  commuting ideals of $L$.
The sum $M_1=I_1+ I_2$ and $M_2=\gamma_{c+2}(M_1)$ are  the sought-for ideals of $L$.
\end{proof}

\section{Groups}\label{s-groups}

Here we  prove the main group theoretic result. Known results reduce the proof to the case where $G$ is a nilpotent group of odd order. Then we firstly apply the Lie ring method similarly to \cite{shu01} to obtain a `weak' bound, depending on $m,n,c$, for the nilpotency class of $[G,\varphi ^{2^{n-1}}]$. Then Theorem~\ref{t2} is used to obtain the required `strong' bound, in terms of $n,c$ only, for the nilpotency class of $[H,\varphi ^{2^{n-1}}]$ for a certain subgroup $H$ of $(m,n,c)$-bounded index.

\begin{proof}[Proof of Theorem~\ref{t1}] Recall that we have a finite group
$G$ admitting an automorphism $\varphi$ of order $2^n$ such that
 $C_G(\varphi ^{2^{n-1}})$ is nilpotent of class $c$, and  $m=|C_G(\varphi )|$ is the
number of fixed points of $\varphi$. We need to prove that $G$ has a
soluble subgroup of  $(m,n,c)$-bounded index that has
$(n,c)$-bounded derived length.

We begin with reduction to the case where $G$ is a nilpotent group
of odd order. The group $G$ has a soluble subgroup of
$(m,n)$-bounded index by Hartley's theorem \cite{har}. Therefore
we can assume from the outset that $G$ is soluble. The quotient
$G/O_{2',2}(G)$ acts faithfully by conjugation on the Frattini
quotient $V=T/\Phi (T)$ of the $2$-group $T=O_{2',2}(G)/
O_{2'}(G)$. By Lemma~\ref{l-fp}(a) we have $|C_V(\varphi )|\leqslant m$.
Therefore the order of $V$ is $(m,n)$-bounded by
Lemma~\ref{l-fpp2}. As a result, the order of  $G/O_{2',2}(G)$ is
also $(m,n)$-bounded.

By  Lemma~\ref{l-fp}(a) we have $|C_T(\varphi )|\leqslant m$.  By Khukhro's
theorem \cite{khu93} on $p$-automorphisms of finite $p$-groups, the group $T$ contains a subgroup $U$ of
$(m,n)$-bounded index that has $n$-bounded derived length. By
 Theorem~\ref{t-char}
this subgroup
$U$ can be assumed to be characteristic in $G/O_{2'}(G)$ and therefore normal and $\varphi$-invariant.

By the Hartley--Turau theorem \cite{ha-tu}, the index of the
$n$-th Fitting subgroup $F_{n}(G)$ in $G$ is $(m,n)$-bounded. By
Lemma~\ref{l-fp}(b) every factor
$Q_i=F_i(O_{2'}(G))/F_{i-1}(O_{2'}(G))$ of the Fitting series of
$F_n(O_{2'}(G))$ admits the action (not necessarily faithful)  of
the automorphism $\varphi$ such that $|C_{Q_i}(\varphi )|\leqslant m$ and
$C_{Q_i}(\varphi ^{2^{n-1}})$ is nilpotent of class at
most $c$.

Suppose that Theorem~\ref{t1} is already proved for the case where
$G$ is a nilpotent group of odd order. Then every $Q_i$ has a
subgroup $R_i$ of $(m,n,c)$-bounded index that is soluble of
$(n,c)$-bounded derived length; by Theorem~\ref{t-char}
this
subgroup can be assumed to be characteristic. Let $\tilde
T=O_{2',2}(G)$, $\tilde U$,  $\tilde Q_i=F_i(O_{2'}(G))$, and
$\tilde R_i$ denote   the inverse images in $G$ of the sections
$T$, $U$, $Q_i$, and $R_i$, respectively. We can set
$$
H=O_{2',2}(G)\cap C_G(\tilde T/\tilde U)\cap
C_G\big(O_{2'}(G)/\tilde Q_n\big)\cap \bigcap _{i=1}^{n-1}
C_G(\tilde Q_i/\tilde R_i). $$
(Here the centralizer of a section $A/B$ is defined naturally as
$C_G(A/B)=\{g\in G\mid [A,g]\leqslant B\}$.) Then $H$ is a subgroup of
$(m,n,c)$-bounded index, since all the quotients $G/C_G(\tilde T/ \tilde
U)$, $G/ C_G\big(O_{2'}(G)/\tilde Q_n\big)$, $G/ C_G(\tilde Q_i/\tilde R_i)$ embed into the automorphism groups of sections of $(m,n,c)$-bounded order. The intersections
of the images of $H$ with the sections $\tilde T/\tilde U$, $O_{2'}(G)/\tilde R_n$, and
$\tilde Q_i/\tilde R_i$ are central in $H$ by construction. Let
$g$ be the derived length of $U$, and $f_i$  the derived length
of $R_i$. Then
\begin{align*}
[H,H]^{(g)}&\leqslant [O_{2',2}(G), \,C_G(\tilde T/\tilde U)]^{(g)}\cap H\\
&\leqslant \tilde U^{(g)}\cap H\leqslant O_{2'}(G)\cap H,
\end{align*}
\begin{align*}
[O_{2'}(G)\cap H, \,O_{2'}(G)\cap H] &\leqslant [O_{2'}(G),
C_G\big(O_{2'}(G)/\tilde Q_n\big)]\cap H
\\
&\leqslant \tilde Q_n
\cap H,
\end{align*}
and
\begin{align*}
\big[\tilde Q_i\cap H, \, \tilde Q_i\cap H\big]^{(f_i)} &\leqslant \big[\tilde Q_i,\,
C_G(\tilde Q_i/\tilde R_i)\big]^{(f_i)}\cap H\\
&\leqslant \tilde R_i^{(f_i)}\cap H\leqslant \tilde Q_{i-1}\cap H,
\end{align*}
where $i=1,2, \ldots, n$ and $\tilde Q_0=1$. It follows
that $H$ is soluble of derived length at most
$$
1+g+1+\sum_{i=1}^n
(1+f_i),
$$
which is an $(n,c)$-bounded number. Thus, $H$ satisfies the
conclusion of Theorem~\ref{t1}, which completes our reduction.

Therefore in what follows we can assume from the outset that $G$
is a nilpotent group of odd order.
We now obtain a `weak' bound, in terms of $m,n,c$,
for the nilpotency class of the subgroup $[G,\varphi ^{2^{n-1}}]$.
For that we consider the
associated Lie ring of $[G,\varphi ^{2^{n-1}}]$, but preliminary lemmas are
stated in terms of abstract Lie rings. To
lighten the notation we denote $\psi =\varphi ^{2^{n-1}}$.
We denote by  $[L, \psi ]$  the additive subgroup generated by $\{-l+l^{\psi}\mid l\in L\}$.

\begin{lemma}\label{l-2.3} If $L$ is a finite metabelian and nilpotent Lie
ring of odd order admitting an automorphism $\varphi$ of order $2^n$
such that $|C_L(\varphi )|=m$, then the ideal $[L, \psi ]+[L,L]$
is nilpotent of $(m,n)$-bounded class.
\end{lemma}

\begin{proof} We actually show that the ideal $[L, \psi ]+[L,L]$
is nilpotent of class at most $1+(m+1)(2^n-1)$.
We can assume from the outset that the ground
ring contains a primitive $2^n$-th root of 1, since the extension of the ground ring  by this root may only
increase the size of the fixed-point subring in terms
of $n$.  As in
  \S\,\ref{s-prelim}, we decompose $L$ into the direct sum of
  analogues of
eigenspaces $L_i$, which serve as components of a
$({\Bbb Z}/2^n{\Bbb Z})$-grading.   Then $[L,\psi ]=L^-$. Therefore we need to
show that any simple homogeneous commutator of length $2+(m+1)(2^n-1)$ in
elements of $L_{\rm odd}$ and $[L,L]$ is trivial. Since $[L,L]$ is
abelian, we can assume that starting from the third place all
entries are in $L_{\rm odd}$, and all these entries can be freely
permuted without changing the commutator. By \cite[Lemma~2.2]{shu01} any
sequence of $2^n-1$ odd numbers can be rearranged to produce
an initial segment with any pre-assigned sum modulo $2^n$. Therefore we can rearrange the $(m+1)(2^n-1)$ entries of our commutator, starting from the third one, so as to produce $m+1$ initial segments in $L_0$. As a result, since $|L_0|=m$,
there will be two different initial segments equal to the same
element in $L_0$. The longer of these two segments can be
substituted instead of the shorter one, then again in the
resulting longer commutator, and so on. Thus the commutator
becomes equal to an ever longer commutator. Since $L$ is nilpotent
by hypothesis, the commutator is equal to 0. \end{proof}

\begin{lemma}\label{l-2.5} Suppose that a finite metabelian and nilpotent
Lie ring $L$ of odd order admits an automorphism $\varphi$ of order
$2^n$ such that $|C_L(\varphi )|=m$ and $C_L(\psi )$ is nilpotent of
class $c$. Then  $\gamma _{g}(L)\leqslant [[L,L], \psi ]$ for some $(m, n, c)$-bounded number $g$. \end{lemma}

\begin{proof} As is Lemma~\ref{l-2.3} we can assume that the ground ring
contains a primitive $2^n$-th root of 1 and $L$ is graded by
analogues of eigenspaces of $\varphi$, so that $[[L,L], \psi
]=[L,L]^-$. Consider the  ideals $J_1=[L,L]+L^-$ and
$J_2=[L,L]+L^+$; then $L=J_1+J_2$. By Lemma~\ref{l-2.3} we have
$\gamma _{f}(J_1)=0$ for some $(m, n)$-bounded number $f$.  Since
$ [L,L]^-$ is an ideal of $J_2$, we have $\gamma _{c+1}(J_2)\leqslant
[L,L]^-$ by hypothesis. We now obtain
\begin{align*}
\gamma _{f+c+1}(J_1+J_2)&\leqslant \gamma _{f}(J_1) + \gamma
_{c+1}(J_2)\\
&\leqslant 0+[L,L]^-,
\end{align*}
as required.
\end{proof}

\begin{proposition}\label{p-2.6} Suppose that a finite Lie ring $L$ of odd order admits an
automorphism $\varphi$ of order $2^n$ such that $|C_L(\varphi )|=m$ and
$C_L(\psi )$ is nilpotent of class $c$. Then the Lie ring
generated by $[L,\psi ]$ is nilpotent of $(m, n, c)$-bounded class.
\end{proposition}

\begin{proof} As before we can assume that the ground ring contains a
primitive $2^n$-th root of 1 and $L$ is graded by analogues of
eigenspaces $L_i$ of $\varphi$. We can obviously assume that
$
L=\langle L^-\rangle =\langle [L,\psi ]\rangle
$.
Consider the lower central series of $L$. The
fixed points of $\varphi$ in its factors are images of the fixed points
in $L$ by Lemma~\ref{l-fp}(b). Therefore there are at most $m$
factors where $\varphi$ is not fixed-point-free. We obtain a series of
$\varphi$-invariant ideals of length at most $2m+1$ each factor of
which either  is central or admits $\varphi$ as a fixed-point-free
automorphism.  The fixed-point subrings of $\psi$ in these factors
are images of subrings of $C_L(\psi )$ by Lemma~\ref{l-fp}(b) and
therefore are nilpotent of class at most $c$. By
Theorem~\ref{t-shu} the factors with fixed-point-free action of $\varphi$
are soluble of $(n,c)$-bounded derived length.  As a result, $L$
is soluble of $(m,n,c)$-bounded derived length. Therefore it is
sufficient to prove by induction on the derived length $d$ of $L=\langle L^-\rangle$ that $L$ is nilpotent of $(d,m,n,c)$-bounded class. If $d=1$, there is nothing to prove, so let $d\geqslant 2$.
 Let $R=L^{(d-2)}$ be the penultimate (metabelian) term of the
derived series of $L$. By the induction hypothesis, $L/[R,R]$ is
nilpotent of $(d-1,m,n,c)$-bounded class. By  Lemma~\ref{l-2.5},
$\gamma _{g}(R)\leqslant [R,R]^-$ for an $(m,n,c)$-bounded number $g$. But
$\gamma _{g}(R)$ is an ideal of $L$, and therefore $[\gamma
_{g}(R), L]=0$, since $[[R, R]^-, L_{\rm odd}]\leqslant L^+$ and $L=\langle
L^-\rangle$. Therefore, in particular, $\gamma _{g+1}(R)=0$. It
remains to apply Hall's Theorem \ref{t-hall}, 
by which $L$ is
nilpotent of  $(d,m,n,c)$-bounded class, as required. \end{proof}

 We now complete the proof of Theorem~\ref{t1}.
Recall that we already have a reduction to the case of a finite
nilpotent group $G$ of odd order admitting an automorphism $\varphi$ of
order $2^n$ such that the fixed-point subgroup $C_G(\psi )$ of the involution $\psi=\varphi ^{2^{n-1}}$ is nilpotent of
class $c$. For $m=|C_G(\varphi )|$ being the number of fixed points of
$\varphi$, we need to prove that $G$ has a soluble subgroup of
$(m,n,c)$-bounded index that has $(n,c)$-bounded derived length.

Recall that  the associated Lie ring $L(D)$ of a group $D$  is defined on the direct sum of lower central factors $L(D)=\bigoplus_i \gamma _i(D)/\gamma _{i+1}(D)$. For $a\in g _{i} (D)$, $b\in \gamma _{j}(D)$, the Lie products  are defined by $[a+\gamma _{i+1} (D),\,b+\gamma _{j+1}(D)]=[a,b]+\gamma _{i+j+1}(D)$ via the group commutator $[a,b]$ on the right and extended to $L(D)$ by linearity. This definition is correct because of the inclusions $[\gamma _i(D),\gamma _j(D)]\leqslant \gamma _{i+j}(D)$.
These inclusions also imply that for any $k$ and any $a_i\in D$,
\begin{equation}\label{e-ass}
[a_1,\dots ,a_k]\gamma _{k+1}(D)=[\bar a_1,\dots \bar a_k],
\end{equation}
where the left-hand side is the image of the group commutator in $\gamma _{k}(D)/\gamma _{k+1}(D)$ and the right-hand side is the commutator in $L(D)$ of the images of $a_i$ in $D/\gamma _2(D)$. In particular, if $D$ is a nilpotent group, then $L(D)$ is a nilpotent Lie ring and its nilpotency class is exactly the same as that of $D$.

Consider the associated Lie  ring $L([G,\psi ])$ of $[G,\psi ]$.
By Lemma~\ref{l-fp}(b) the induced automorphism $\varphi$ denoted by
the same letter has the same number $|C_{L([G,\psi ])}(\varphi 
)|=|C_{[G,\psi ]}(\varphi )|\leqslant m$ of fixed points. Since
$C_{L([G,\psi ])}(\psi  )$ is the sum of the images of subgroups
of $C_{[G,\psi ]}(\psi  )$ by Lemma~\ref{l-fp}(b), it is easy to
see that $C_{L([G,\psi ])}(\psi  )$ is also nilpotent of class at
most $c$.  By Proposition~\ref{p-2.6} we obtain that $L([G,\psi
])$, and therefore also $[G,\psi ]$, is nilpotent of
$(m,n,c)$-bounded class $k$. However, our aim is a  subgroup of
bounded index with derived length `strongly' bounded, in terms  of
$n$ and $c$ only, independently of $m=|C_G(\varphi )|$. We will
achieve this goal by applying Theorem~\ref{t2} to find a subgroup
$H$ of $(m,n,c)$-bounded index such that $[H,\varphi ^{2^{n-1}}]$ is
nilpotent of $(n,c)$-bounded class.

We extend the ground ring by a primitive $2^n$-th root of unity
$\omega$ forming  $L=L([G,\psi ])\otimes _{{\Bbb Z} }{\Bbb Z} [\omega ]$. Then $L= L_0
\oplus L_1 \oplus \dots \oplus L_{2^n-1}$ is a $({\Bbb Z}/2^n{\Bbb Z})$-graded
Lie ring with grading components $L_i$ --- analogues of eigenspaces of
$\varphi$ --- satisfying $[L_s, L_t]\subseteq L_{s+t\,({\rm mod}\,2^n)}$. As
usual, the Lie ring $L([G,\psi ])$ is considered to be embedded in
$L$ as $L([G,\psi ])\otimes 1$. The Lie ring $L$ is nilpotent of
the same nilpotency class $k$. We also have $|C_L(\varphi )|\leqslant
m^{2^n}$ and $C_L(\psi )$ is nilpotent of class at most $c$.

By Proposition~\ref{p1} the Lie ring $L$ has a nilpotent ideal $B$
of $(n,c)$-bounded class $h$ such that $B\cap L^-$ has
$(m,n)$-bounded index in $L^-$. Since $[G,\psi]$ is generated by
elements $x$ such that $x^{\psi}=x^{-1}$, it follows that $L$ is
generated by elements $l$ such that $l^{\psi}=-l$, that is,
$L=\langle L^-\rangle=L^-+[L,L]$.  Then $M=B +[L,L]$ is an ideal
of $(m,n)$-bounded index in $L$. The nilpotency class of $M=B +[L,L]$ is
strictly smaller than the nilpotency class $k$ of $L$, as long as $k$  was higher than $h$. Indeed,
consider any commutator of weight $k$ in elements of $B\cup [L,L]$. If it involves at least one element of $[L,L]$, then it  clearly belongs to $\gamma _{k+1}(L)=0$; otherwise it belongs to $\gamma _{k}(B)$, which is trivial when $k>h$.

Consider $T=M\cap L([G,\psi ])$, which is an
ideal of the Lie ring $L([G,\psi ])$ containing $\gamma _2(L([G,\psi ]))$.
Taking the `full inverse image' of $T$  modulo  $\gamma _2([G,\psi ])$ we obtain a subgroup $G_1$ of $(m,n)$-bounded
index in  $[G,\psi ]$.  As long as $h<k$, the  nilpotency class of $G_1$ is strictly
smaller than the nilpotency class $k$ of $[G,\psi ]$. Indeed,
consider any commutator $[a_1,\dots ,a_k]$ of weight $k$ in elements  $a_i\in G_1$. Since $\gamma _{k+1}([G,\psi ])=1$, by formula \eqref{e-ass} we have
$$
[a_1,\dots ,a_k]=[\bar a_1,\dots ,\bar a_k],
$$
where $\bar a_i$ is the image of $a_i$ in $[G,\psi ]/\gamma _2([G,\psi ])$. By construction, $\bar a_i\in T$ and therefore the Lie ring commutator on the right is equal to $0$ if $k>h$, which also means that  $
[a_1,\dots ,a_k]=1$.

By the Bruno--Napolitani theorem \cite[Lemma~3]{brna} (see also Theorem~\ref{t-char}), there is also a characteristic subgroup of $[G,\psi ]$ that has  $(m,n)$-bounded index in $[G,\psi ]$ and is nilpotent of class at most $k-1$.  Changing notation we denote this
subgroup again by $G_1$, which is now normal in $G$ and $\varphi$-invariant. Then the product
$G_2=G_1C_G(\psi )$ is a $\varphi$-invariant subgroup of $G$ of
$(m,n)$-bounded index (the latter because $G=[G,\psi ]C_G(\psi )$ by Lemma~\ref{l-fp}(b)),  and the nilpotency class
of $[G_2, \psi]\leqslant G_1$ is strictly smaller than $k$. We can now apply
the same arguments to $G_2$ and so on, at each step obtaining a
$\varphi$-invariant subgroup $G_{2i}$ containing $C_G(\psi )$ and having
$(m,n)$-bounded index in $G_{2i-2}$ such that $[G_{2i},\psi ]$
has nilpotency class strictly smaller than that of $G_{2i-2}$ --- as
long as the latter remains greater than the $(n,c)$-bounded number
$h$ given by Propositon~\ref{p1}. The number of these steps is
$(m,n,c)$-bounded, since the nilpotency class of $[G,\psi ]$ is
$(m,n,c)$-bounded. As a result, we arrive at a subgroup $H$ of
$(m,n,c)$-bounded index in $G$ such that $[H,\psi ]$ is nilpotent
of $(n,c)$-bounded class at most $h$. Since $C_H(\psi )$ is nilpotent of class
at most $c$ by hypothesis, this subgroup $H$ is soluble of
$(n,c)$-bounded derived length. By Theorem~\ref{t-char} there is also a characteristic subgroup of $(m,n,c)$-bounded index in $G$ which has the same derived length as $H$. \end{proof}

\begin{remark}\label{r1}
The condition in the theorem that $C_G(\varphi ^{2^{n-1}})$ is nilpotent of class  $c$ can be weakened to requiring all Sylow subgroups of
$C_G(\varphi ^{2^{n-1}})$ to be nilpotent of class at most $c$. Indeed, that condition is not used in the reduction at the beginning of the section to the case $G=O_{2'}(G)$. After that, as we saw, it is sufficient to consider the factors $Q_i$ of the Fitting series of $O_{2'}(G)$. If all  Sylow subgroups of
$C_G(\varphi ^{2^{n-1}})$ are nilpotent of class at most $c$, then $C_{Q_i}(\varphi ^{2^{n-1}})$ is nilpotent of class at most $c$ for every $i$ and we find ourselves under the hypotheses of Theorem~\ref{t1}.

Similarly,  in Corollary~\ref{c1} the condition that $C_G(g ^{2^{n-1}})$ is nilpotent of class  $c$ can be weakened to requiring all nilpotent subgroups of
$C_G(g ^{2^{n-1}})$ to be nilpotent of class at most $c$, because when applying the inverse limit argument to a system of finite subgroups containing $g$, we would be able to use the aforementioned stronger version of Theorem~\ref{t1}.
\end{remark}

\section*{Acknowledgments} This work was supported by CNPq-Brazil. The first author thanks  CNPq-Brazil
and the University of Brasilia for support and hospitality that he
enjoyed during his visits to Brasilia. The second
author was supported  by the Russian Foundation for Basic
Research, project no. 13-01-00505. 

The authors thank the referee for careful reading of the paper and several helpful comments. 

\end{document}